\documentclass[a4paper, 11pt]{article}
\usepackage{graphicx}
\usepackage{subcaption}
\usepackage{float}
\usepackage{amsmath}
\usepackage{amsthm}
\usepackage{amssymb}
\usepackage[title]{appendix}
\usepackage{multirow}
\usepackage[noblocks]{authblk}
\usepackage{epsfig}
\usepackage{geometry}
\usepackage{fancyhdr}
\usepackage{url}
\usepackage{xcolor}
\usepackage{dsfont}
\usepackage{bm}
\usepackage{tikz}
\usepackage{nicefrac}
\usepackage{mathtools}
\usepackage{cite}
\usepackage{authblk}

\usepackage{algorithm2e}
\SetArgSty{textnormal}

\usetikzlibrary{arrows,decorations.markings}

\makeatletter
\providecommand*{\diff}%
{\@ifnextchar^{\DIfF}{\DIfF^{}}}
\def\DIfF^#1{%
\mathop{\mathrm{\mathstrut d}}%
\nolimits^{#1}\gobblespace}
\def\gobblespace{%
\futurelet\diffarg\opspace}
\def\opspace{%
\let\DiffSpace\!%
\ifx\diffarg(%
\let\DiffSpace\relax
\else
\ifx\diffarg[%
\let\DiffSpace\relax
\else
\ifx\diffarg\{%
\let\DiffSpace\relax
\fi\fi\fi\DiffSpace}

\providecommand\given{} % so it exists
\newcommand\SetSymbol[1][]{
   \nonscript\,#1\vert \allowbreak \nonscript\,\mathopen{}}
\DeclarePairedDelimiterX\Set[1]{\lbrace}{\rbrace}%
 { \renewcommand\given{\SetSymbol[\delimsize]} #1 }

\theoremstyle{plain}

\newcommand{\eref}[1]{Eq.\,(\ref{#1})}
\newcommand{\fref}[1]{Fig.\,\ref{#1}}

\newcommand{\sep}{\,\,\,\,}
\newcommand{\sepand}{\,\,\text{and}\,\,}

\newcommand{\tp}[1]{#1\raisebox{1.15ex}{$\scriptscriptstyle\!\text{T}$}}

\newcommand{\normi}[1]{{\left\lVert#1\right\rVert}_{\infty}}

\newcommand{\tmadd}{\oplus}
\newcommand{\tmmul}{\odot}
\newcommand{\tmr}{\mathbb{R}_{\text{max}}}
\newcommand{\cone}{\mathcal{C}}

\newcommand{\defeq}{\vcentcolon=}

\DeclareMathOperator*{\argmax}{argmax}
\DeclareMathOperator*{\argmin}{argmin}

\renewcommand{\qed}{\hfill\blacksquare}

\newtheorem{theorem}{Theorem}[section]
\newtheorem{lemma}[theorem]{Lemma}

\renewenvironment{proof} {\par{\it Proof.} \ignorespaces} {\par\medskip}
\theoremstyle{definition}
\newtheorem{definition}[theorem]{Definition}

\theoremstyle{remark}

\numberwithin{equation}{section}

\title{Extreme rays of the $\ell^\infty$-nearest ultrametric tropical polytope}
\author{Luyan Yu\footnote{luyan.yu@utexas.edu 
}}
\affil{Department of Physics, University of Texas at Austin}

\date{}                     %% if you don't need date to appear
\begin{document}
\maketitle
\thispagestyle{fancy}
\fancyhead{}
\lhead{}
\chead{}
\rhead{}
\lfoot{}
\cfoot{\thepage}
\rfoot{}
\renewcommand{\headrulewidth}{0pt}
\renewcommand{\footrulewidth}{0pt}

\begin{abstract}
The set of ultrametrics on $[n]$ nodes that are $\ell^\infty$-nearest to a given dissimilarity map forms a $(\max,+)$ tropical polytope. Previous work of Bernstein has given a superset of the set containing all the phylogenetic trees that are extreme rays of this polytope. In this paper, we show that Bernstein's necessary condition of tropical extreme rays is sufficient only for $n=3$ but not for $n\geq 4$. Our proof relies on the exterior description of this tropical polytope, together with the tangent hypergraph techniques for extremality characterization. The sufficiency of the case $n=3$ is proved by explicitly finding all extreme rays through the exterior description. Meanwhile, an inductive construction of counterexamples is given to show the insufficiency for $n\geq 4$.
\end{abstract}

\section{Introduction}
Phylogenetic trees are used as a compact way of representing the evolutionary history of systems such as biological species. In the case of biological evolution, the leaves of a tree are known species and the internal nodes are hypothetical ancestors. The higher up in the tree the common ancestor of two species is, the further evolutionary relationships these two species have. A tree with edge lengths or internal node weights specified induces a metric on the leaves, which measures the discrepancy among entities under consideration. An ultrametric $\delta$ is a metric satisfying the additional strengthened triangle inequality $\delta_{ik} \leq \max(\delta_{ij},\delta_{jk})$ for all $i,j \text{ and } k$, or equivalently, $\max(\delta_{ij},\delta_{jk},\delta_{ik})$ achieved at least twice. An ultrametric encodes the topology of the underlying phylogenetic tree and one can recover it when needed.

In practice, one is usually given dissimilarity maps that measure the pairwise discrepancy among entities. Most of the times such dissimilarity maps are not exact but presumably approximate to some ultrametrics. The reconstruction of the ultrametrics from dissimilarity maps is of great interest across many scientific areas \cite{takezaki1996genetic, saitou1987neighbor, strimmer1996quartet}. Although many choices of distance measure between two dissimilarity maps exist, the $\ell^\infty$-distance becomes a natural and preferable choice since it blends well into the $(\max,+)$ tropical algebra together with the defining property of ultrametrics. Recent advancements also imply deeper relations between phylogenetic trees and tropical geometry: \cite{ardila2006bergman} established the homeomorphism between Bergman fan of the graphical matroid of the complete graph to the space of phylogenetic trees; \cite{lin2018tropical} gave the space of phylogenetic trees a statistical treatment framework based on tropical geometry and \cite{manon2012phylogenetic} discussed the relationship between the space of weighted phylogenetic trees and the tropical varieties of flag variety. Also, an interesting application of the tropical principal component analysis to phylogenetic trees can be found in \cite{yoshida2017}.

It was shown by Bernstein in \cite{bernstein2017} that the set of ultrametrics that are $\ell^\infty$-nearest to a given dissimilarity map $d$ forms a tropical polytope. A tropical polytope admits two equivalent descriptions --- by tropical halfspaces (exterior description) and by extreme rays (interior description). In pursuit of a direct method to find extreme rays, Bernstein proposed a procedure as well as a criterion that can generate a set of ultrametrics $\mathcal{B}_n(d)$ and showed that $\mathcal{B}_n(d) \supseteq \mathcal{E}_n(d)$, where $\mathcal{E}_n(d)$ is the set of extreme rays. But it was not clear if $\mathcal{B}_n(d) = \mathcal{E}_n(d)$. In this paper, we prove the following theorem.
\begin{theorem}\label{maintheorem}
The necessary extremality characterization by Bernstein is sufficient only when $n=3$, i.e.,
$\mathcal{B}_3(d) = \mathcal{E}_3(d)$;\, $\mathcal{B}_n(d) \supsetneq \mathcal{E}_n(d),\, \forall n\geq 4$.
\end{theorem}
This means that for $n\geq 4$, some ultrametrics that are not extreme rays may also be produced and pass the criterion. Thus, a direct procedure to produce all and only extreme rays is yet still unknown. The proof of $n=3$ is achieved by using the exterior description and the tangent directed hypergraph techniques \cite{allamigeon2013computing} to characterize the extremality. Meanwhile, the disproof of $n\geq 4$ is done by giving an explicit counterexample of $n=4$ and then using it to generate larger counterexamples inductively. Also, the $n=8$ example given in \cite{bernstein2017} based on biological data is examined and it turns out it is also a counterexample to the sufficiency.

This paper is organized as follows. In Section \ref{sec:prelim} we briefly introduce the background of tropical interior/exterior description, as well as phylogenetic trees and ultrametrics. In Section \ref{sec:extdesc} we give the exterior description of $\ell^\infty$-nearest ultrametrics polytope. In Section \ref{sec:n3} we prove Theorem \ref{maintheorem} for the case $n=3$ and in Section \ref{sec:counter} for the case $n\geq 4$. Finally, we summarize in Section \ref{sec:disc}.

\section{Background}\label{sec:prelim}
\subsection{Tropical double description}
For a complete review of tropical geometry, readers should refer to \cite{maclagan2015introduction, butkovivc2010max, baccelli1992synchronization, speyer2009tropical}.
In $(\max,+)$ tropical geometry, tropical addition $\tmadd$ is defined to be the \textit{max} and tropical multiplication $\tmmul$ to be the usual \textit{addition}. These two operations together with the set $\tmr \defeq \mathbb{R}\cup \{-\infty\}$ forms the max tropical semiring $(\tmr, \tmadd, \tmmul)$. Many concepts of classical geometry such as polyhedral cone and polyhedron have tropical analogs. In the following discussion, we work in the $d$-dimensional ambient space $\tmr^d$ and use the usual boldface vector notation for the elements: $\bm{x}=\{x_1,\cdots,x_d\} \in \tmr^d$. The tropical inner product is defined as: for $\bm{x},\bm{y} \in \tmr^d$,
$$\bm{x}\cdot \bm{y} = \max_{1\leq i \leq d}(x_i + y_i).$$
Let $\argmax(\bm{x}\cdot\bm{y}) \subset [d]$ be the set of indices that achieve the maximum. Similarly, the matrix-vector multiplication is defined as: for $A \in \tmr^{n\times d}$ and $\bm{x} \in \tmr^d$,
$$A\bm{x} \in \tmr^n \sepand (A\bm{x})_k = A_k \cdot \bm{x} = \max_{1\leq i \leq d}(A_{ki} + x_i),$$
where $k\in [n]$ and $A_k$ denotes the $k$th row of matrix $A$.
Tropical halfspace and polyhedral cone are direct analogs of their classical versions:
\begin{definition}[Tropical halfspace]
	 Given $\bm{a},\bm{b} \in \tmr^d$, the tropical halfspace is the set
	 \begin{equation}\label{prelim:eq1}
	 \Set{\bm{x}\in\tmr^d  \given \bm{a}\cdot\bm{x}\leq \bm{b}\cdot\bm{x} }.
	 \end{equation}
\end{definition}
\begin{definition}[Tropical polyhedral cone]
	Given $A, B \in \tmr^{n\times d}$, the tropical polyhedral cone $\cone$ is the set
	\begin{equation}\label{prelim:eq2}
    \cone \defeq \Set{\bm{x}\in\tmr^d  \given A \bm{x} \leq B \bm{x}}.
	\end{equation}
\end{definition}
\noindent
The above definition of tropical polyhedral cone is the exterior description by the set of defining halfspaces. It can also be described by a set of points called extreme rays defined as follows. The cone described by the extreme rays will be referred to as the interior description.
\begin{definition}[Extreme rays]
Let $\cone$ be a tropical polyhedral cone. The set of extreme rays of $\cone$ is the minimal set of points $\{\bm{u}_{j}\} \subseteq \cone$ such that
	\begin{equation}\label{prelim:eq3}
    \cone = \Set{ \bm{x} \in \tmr^d\given \bm{x} = \bigoplus_j \lambda_j \bm{u}_j, \lambda_j \in \tmr}.
	\end{equation}
\end{definition}
\noindent
Given the exterior description of a tropical polyhedral cone, the state-of-the-art algorithm to compute its interior description was given by \cite{allamigeon2014complexity}.
A tropical polyhedron of $\tmr^d$ is an affine version of the tropical polyhedral cone:
\begin{definition}[Tropical polyhedron]
	Given $A, B \in \tmr^{n\times d}$ and $\bm{c},\bm{d}\in \tmr^n$, the tropical polyhedral is the set
	\begin{equation}\label{prelim:eq4}
	\mathcal{P} \defeq \Set{\bm{x}\in\tmr^d  \given A \bm{x} \tmadd \bm{c} \leq B \bm{x} \tmadd \bm{d}}.
	\end{equation}
\end{definition}
\noindent
It is known \cite{gaubert2007minkowski} that a tropical polyhedron in $\tmr^d$ can be transformed into a tropical polyhedral cone in $\tmr^{d+1}$ by a homogenization process which we will show in an example in Section \ref{sec:extdesc}.

\subsection{Phylogenetic trees and ultrametrics}
Let $X$ be a finite set with $n$ elements with labels $i \in [n]$. First we define the dissimilarity map. 
\begin{definition}[Dissimilarity map]
	A dissimilarity map on $X$ is a function $d : X \times X \rightarrow \mathbb{R}$ such that $d_{ii}=0$ and $d_{ij} = d_{ji},\sep \forall i,j \in X$.  A dissimilarity map can be expressed as a symmetric matrix with zero diagonal. We also denote this matrix by $d\in \mathbb{R}^{[n]\choose 2}$.
\end{definition}
\noindent
Given two dissimilarity maps $d_1, d_2$ on $X$, we define the $\ell^\infty$ distance as $$\normi{d_1-d_2}=\max_{i,j\in [n]} |(d_1)_{ij}-(d_2)_{ij}|.$$

Phylogenetic trees come with various types. In this paper we will be focusing on the rooted tree. The weight of the tree can be either defined on the edges or nodes. Here we follow the definitions in \cite{bernstein2017} that assigns weights on the internal nodes of the tree. The rooted tree is defined as follows.
\begin{definition}[Rooted tree]
	A rooted tree $T$ on $X$ is a tree with leaf set $X$ and one interior node is designated as the root, denoted as $\text{root}(T)$.
\end{definition}
\noindent
The descendants of a vertex $v$ in $T$, denoted as $\text{Des}_T(v)$, is all node $v'\neq v$ such that the unique path from $v'$ to $\text{root}(T)$ contains $v$. Denote by $T^\circ$ the set of interior nodes of $T$. Let $\alpha : T^\circ \rightarrow \mathbb{R}$ be a weighting of the internal nodes of $T$. The pair $(T,\alpha)$ induces a dissimilarity map $d_{T,\alpha}$ on $X$ defined by $(d_{T,\alpha})_{ij} = \alpha(v)$ where $v\in T^\circ$ is the node nearest to $\text{root}(T)$ in the unique path from $x_i$ to $x_j$.

\begin{definition}[Ultrametric]
	Given dissimilarity map $\delta$ on $X$, $\delta$ is an ultrametric if it is a metric and satisfies the strengthened triangle inequality
	\begin{equation}\label{prelim:eq5}
	\forall i,j,k \in [n],\, \delta_{ik} \leq \max(\delta_{ij},\delta_{jk}).
	\end{equation}
	Denote by $\mathcal{U}_n$ the set of ultrametrics on $n$ elements.
\end{definition}
\noindent
Given an ultrametric $\delta$ on $X$, we can construct a rooted tree $T$ with weighting $\alpha$ which induces $u$. Also, such construction is unique if $\alpha(u) < \alpha(v), \forall u \in \text{Des}_T(v)$.
\begin{definition}[$\ell^\infty$-nearest ultrametrics]
	Given dissimilarity map $d$, the set of ultrametrics that are $\ell^\infty$-nearest to $d$ is
	\begin{equation}\label{prelim:eq6}
	\argmin_{\delta\in \mathcal{U}_n} \normi{\delta-d}.
	\end{equation}
\end{definition}
\noindent
It was shown by Proposition 3.2 in \cite{bernstein2017} that this set is a tropical polytope. Let $\mathcal{E}_n(d)$ denote the set of all extreme rays of this polytope.

\section{Exterior description of $\ell^\infty$-nearest ultrametrics}\label{sec:extdesc}
The exterior description can be obtained by considering the definition of ultrametrics \eref{prelim:eq5} and the condition for it to be $\ell^\infty$-nearest.

Viewed as a symmetric matrix with zero diagonal, an ultrametric $\delta$ on $n$ elements consists of $n(n-1)/2$ entries. We represent it as a vector $\bm{\delta}$, with these $n(n-1)/2$ elements arranged in lexicographical order. For example, an ultrametric on 3 elements is arranged as $\bm{\delta} = (\delta_{12},\delta_{13},\delta_{23})$. If we look at \eref{prelim:eq5}
\begin{equation}\label{extdesc:eq1}
\forall i,j,k\in [n],\,\, \delta_{ik}\leq \max(\delta_{ij},\delta_{jk}),\tag{\ref{prelim:eq5}}
\end{equation}
we will get $3\times {n \choose 3} = n(n-1)(n-2)/2$ tropical halfspaces.

Then we consider the condition for it to be $\ell^\infty$-nearest. Suppose we are given the dissimilarity map $d=\{d_{ij}\}$, and further, the nearest distance $q = \normi{d-\delta'}$ where $\delta'$ is any $\ell^\infty$-nearest ultrametric. For given $d$, such distance $q$ exists and is unique and it can be calculated by a specific instance of $\ell^\infty$-nearest ultrametric. The following algorithm \cite{chepoi2000approximation} finds one $\ell^\infty$-nearest ultrametric.
\begin{theorem}[Chepoi \& Fichet \cite{chepoi2000approximation}]\label{thm:oneultra}
	Let $d$ be a dissimilarity map on $n$ elements. The map $\delta^*$ constructed in the following way is an ultrametric that is $\ell^\infty$-nearest to $d$.
	\begin{itemize}
		\item Let $K_n=(X,E)$ be a complete graph with edge weights $d_{ij}$ and $T$ be its minimum spanning tree;
		\item define $d^*_{ij}$ to be the maximal edge of the unique path in $T$ connecting $i$ and $j$;
		\item let $q=\normi{d^*-d}/ 2$ and $\delta^* = d^* + q \bm{1}$ where $\bm{1}_{ij}=1,\forall i\neq j\in [n]$.
	\end{itemize}
\end{theorem}
\noindent
Minimal spanning tree can be easily computed by algorithms such as Kruskal's algorithm. It is worth noting that $d^*_{ij}$ can also be written as the minimal bottleneck between $i$ and $j$:
\begin{equation}\label{extdesc:eq2}
d^*_{ij} = \min_{\text{path $P$ from $i$ to $j$}} \left( \max_{(u,v)\in P} d_{uv}\right).
\end{equation}
\noindent
The $\ell^\infty$-nearest condition gives
\begin{equation}\label{extdesc:eq3}
\forall i\neq j \in [n],\,\, |\delta_{ij}-d_{ij}| \leq q,
\end{equation}
or written separately,
\begin{equation}\label{extdesc:eq4}
\forall i\neq j \in [n],\,\, \delta_{ij}-q \leq d_{ij} \sepand d_{ij} \leq \delta_{ij} + q.
\end{equation}
This gives us $2\times {n \choose 2} = n(n-1)$ inequalities.

These two sets of inequalities \eref{extdesc:eq1} and \eref{extdesc:eq4} together give us the exterior tropical polyhedron description of the $\ell^\infty$-nearest ultrametrics. In total there are $n^2(n-1)/2$ inequalities so the matrix of the tropical polyhedron is of dimension $n^2(n-1)/2 \times n(n-1)/2$. Since the exterior to interior conversion algorithm works on tropical cones, we need to homogenize the tropical polyhedron. By adding auxiliary variable $\xi$ to both sides of \eref{extdesc:eq1} and \eref{extdesc:eq4} and making replacement $\delta_{ij} + \xi \rightarrow \tilde{\delta}_{ij}$, we have
\begin{equation}\label{extdesc:eq5}
\begin{aligned}
&\tilde{\delta}_{ij}\leq \max(\tilde{\delta}_{jk},\tilde{\delta}_{ik}),\\
&\tilde{\delta}_{ij}-q \leq  \xi + d_{ij} \sepand \xi + d_{ij}  \leq \tilde{\delta}_{ij} + q.
\end{aligned}
\end{equation}
\eref{extdesc:eq5} is the exterior tropical cone description of the $\ell^\infty$-nearest ultrametrics.

Allamigeon et al. gives an algorithm that can compute all the extreme rays from the exterior descriptions \cite{allamigeon2014complexity}. Suppose by using this algorithm we find $p$ extreme rays: $\tilde{\bm{\delta}}_1, \cdots, \tilde{\bm{\delta}}_p$. Then any $\ell^\infty$-nearest ultrametrics can be generated by taking tropical linear combination of the extremes
\begin{equation}\label{extdesc:eq6}
\tilde{\bm{\delta}} =  \bigoplus_{i=1}^p \lambda_i  \tilde{\bm{\delta}}_i, \lambda_i \in \tmr,
\end{equation}
followed by a tropical normalization that subtracts the value of the auxiliary variable.

\subsection{A small example ($n=3$)}
We here present explicitly a small example of $n=3$ to demonstrate the above procedure. In the case of $n=3$, $\bm{\delta} = (\delta_{12},\delta_{13},\delta_{23})$.
The ultrametric condition gives
\begin{equation}\label{extdesc:ex:eq1}
\left(
\begin{array}{ccc}
0&& \\
&0& \\
&&0 \\
\end{array}
\right)\bm{\delta} \leq \left(
\begin{array}{ccc}
&0&0 \\
0&&0 \\
0&0& \\
\end{array}
\right)\bm{\delta},
\end{equation}
where blank spots are tropical zero, i.e., $-\infty$. And the nearest condition gives
\begin{equation}\label{extdesc:ex:eq2}
 \left(
\begin{array}{ccc}
\\
\\
\\
-q&&\\
&-q&\\
&&-q\\
\end{array}
\right)\bm{\delta}
\tmadd
 \left(\begin{array}{c}
d_{12}\\
d_{13}\\
d_{23}\\
\\
\\
\\
\end{array}\right)
\leq
 \left(
\begin{array}{ccc}
q&&\\
&q&\\
&&q\\
\\
\\
\\
\end{array}
\right)\bm{\delta}
\tmadd
\left(\begin{array}{c}
\\
\\
\\
d_{12}\\
d_{13}\\
d_{23}\\
\end{array}\right).
\end{equation}
Let $\tilde{\bm{\delta}} = (\xi, \tilde{\delta}_{12},\tilde{\delta}_{13}, \tilde{\delta}_{23})$. After homogenization we get
\begin{equation}\label{extdesc:ex:eq3}
\left(
\begin{array}{cccc}
&0&& \\
&&0& \\
&&&0 \\
d_{12}&&&\\
d_{13}&&&\\
d_{23}&&&\\
&-q&&\\
&&-q&\\
&&&-q\\
\end{array}
\right) \tilde{\bm{\delta}}
\leq
\left(
\begin{array}{cccc}
&&0&0 \\
&0&&0 \\
&0&0& \\
&q&&\\
&&q&\\
&&&q\\
d_{12}&&&\\
d_{13}&&&\\
d_{23}&&&\\
\end{array}
\right) \tilde{\bm{\delta}}.
\end{equation}

To be more concrete, let the given dissimilarity map to be
\begin{equation}\label{extdesc:ex:eq4}
d = \left(
\begin{array}{ccc}
0&2&4 \\
2&0&8 \\
4&8&0 \\
\end{array}
\right),
\end{equation}
i.e., $d_{12}=2, d_{13}=4, d_{23}=8$. The one $\ell^\infty$-nearest ultrametric given by Theorem \ref{thm:oneultra} is
\begin{equation}\label{extdesc:ex:eq5}
\delta^* = \left(
\begin{array}{ccc}
0&4&6 \\
4&0&6 \\
6&6&0 \\
\end{array}
\right),
\end{equation}
corresponding to the tree shown in \fref{fig:oneultra}. Along with the ultrametric, we also get $q = \normi{d-\delta^*} = 2$.
\begin{figure}[!htpb]
	\centering
	\includegraphics{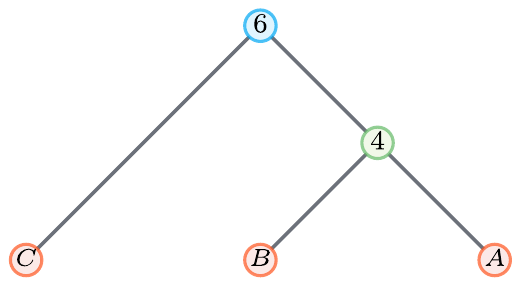}
	\caption{One ultrametric returned by Theorem \ref{thm:oneultra}}
	\label{fig:oneultra}
\end{figure}

\noindent
\eref{extdesc:ex:eq3} realizes to be
\begin{equation}\label{extdesc:ex:eq6}
\left(
\begin{array}{cccc}
&0&& \\
&&0& \\
&&&0 \\
2&&&\\
4&&&\\
8&&&\\
&-2&&\\
&&-2&\\
&&&-2\\
\end{array}
\right) \tilde{\bm{\delta}}
\leq
\left(
\begin{array}{cccc}
&&0&0 \\
&0&&0 \\
&0&0& \\
&2&&\\
&&2&\\
&&&2\\
2&&&\\
4&&&\\
8&&&\\
\end{array}
\right) \tilde{\bm{\delta}}.
\end{equation}
Using the extreme computing algorithm we find two (normalized) extreme rays:
\begin{equation}\label{extdesc:ex:eq7}
\tilde{\bm{\delta}} = \left(
\begin{array}{c}
0\\
0\\
6\\
6\\
\end{array}
\right) \sepand
\tilde{\bm{\delta}}=
\left(
\begin{array}{c}
0\\
4\\
6\\
6\\
\end{array}
\right) \longrightarrow
\bm{\delta} = \left(
\begin{array}{c}
0\\
6\\
6\\
\end{array}
\right) \sepand
\bm{\delta}=
\left(
\begin{array}{c}
4\\
6\\
6\\
\end{array}
\right).
\end{equation}
They corresponds to the two trees shown in \fref{fig:exextreme} respectively.
\begin{figure}[!htbp]
	\centering
	\begin{subfigure}[b]{0.48\textwidth}
        \centering
		\includegraphics{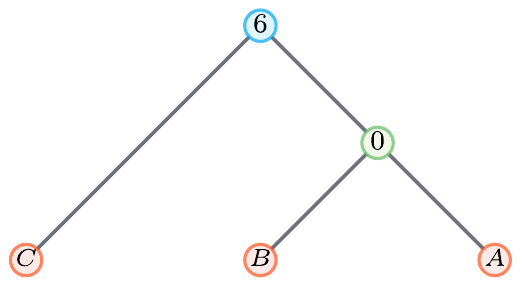}
	\end{subfigure}
	\hfill
	\begin{subfigure}[b]{0.48\textwidth}
        \centering
		\includegraphics{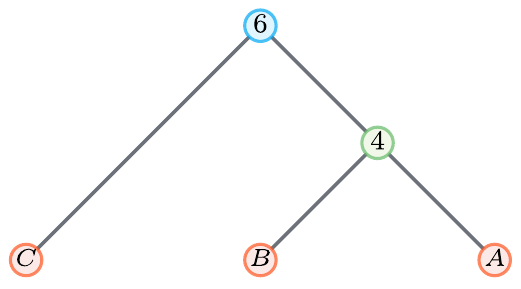}
	\end{subfigure}
	\caption{Extreme ultrametrics $\ell^\infty$-nearest to the dissimilarity map given by \eref{extdesc:ex:eq4}}
	\label{fig:exextreme}
\end{figure}

Finally, any $\ell^\infty$-nearest ultrametrics can be written as the tropical linear combination as in \eref{extdesc:eq6} of the above two ultrametrics.

\section{Proof of Theorem \ref{maintheorem} for the case $n=3$}\label{sec:n3}
Bernstein\cite{bernstein2017} gives a necessary condition to be extreme. It generates a set containing all extreme rays and some non-extreme ones, within which the extreme ones should satisfy the criterion. Before stating the theorem, we need the following definition for the mobility of nodes.
\begin{definition}[Mobility]
	Let $d$ be a dissimilarity map on $X$ and let $\delta$ be an ultrametric that is closest to $d$ in the $\ell^\infty$-norm. Let $T$ be a resolution of the topology of $\delta$ and let $\alpha$ be the internal nodes weighting such that $d_{T,\alpha}=\delta$. An internal node $v$ of $T$ is said to be \textit{mobile} if there exists an ultrametric $\hat{\delta} \neq \delta$, expressible as $\hat{\delta} = d_{T,\hat{\alpha}}$ such that
	\begin{itemize}
		\item $\hat{\delta}$ is also nearest to $\delta$ in the $\ell^\infty$-norm,
		\item $\hat{\alpha}(x) = \alpha(x), \forall x \in T^\circ, x\neq v$, and
		\item $\hat{\alpha}(v) \leq \alpha(v)$.
	\end{itemize}
	In this case, we say that $\hat{\delta}$ is obtained from $\delta$ by sliding $v$ down. If moreover $v$ is no longer mobile in $d_{T,\hat{\alpha}}$, i.e., if $\hat{\alpha}(v) = \max\{\alpha(y) :y \in \text{Des}_T(v)\}$, or $\hat{\alpha}(v)$ is the minimum value such that $d_{T,\hat{\alpha}}$ is nearest to $d$ in the $\ell^\infty$-norm, then we say that $\hat{\delta}$ is obtained from $\delta$ by sliding $v$ all the way down.
\end{definition}
\noindent
Intuitively, the mobility of an internal node indicates the potential to decrease its weight without violating the conditions for ultrametric and $\ell^\infty$-nearest. 
\begin{theorem}[Bernstein \cite{bernstein2017}] \label{thm:charac}
	Let $d$ be a dissimilarity map on $X$. Let $\delta^*$ be the ultrametric generated by Theorem \ref{thm:oneultra}. Let $S_0=\{\delta^*\}$, and for each $i\geq 1$ define $S_i$ to be the set of ultrametrics obtained from some $u\in S_{i-1}$ by sliding a mobile internal node of a resolution of the topology of $u$ all the way down. Define $$\mathcal{B}_n(d) = \{\delta \in \cup_i S_i \, | \, \delta \text{ has at most one mobile internal node}\}.$$
Then $\mathcal{B}_n(d) \supseteq \mathcal{E}_n(d)$.
\end{theorem}
\noindent

In this proof, we will use the tangent directed hypergraph techniques from \cite{allamigeon2013computing}. We state some definitions and a theorem we will use in the following.
\begin{definition}[Tangent directed hypergraph]
	The tangent directed hypergraph at $\bm{v}$, denoted by $\mathcal{G}(\bm{v},\cone)$, consisting of the hyperarcs $(\argmax(B_k\bm{v}), \argmax(A_k\bm{v}))$ for every $k\in [p]$ such that $A_k \bm{v} =B_k\bm{v} >-\infty$.
\end{definition}

\noindent
As an example, we construct the tangent directed hypergraphs of \eref{extdesc:ex:eq6} with respect to non-extreme vector $\bm{\delta} = \tp{(0,2,6,6)}$ and extreme one $\bm{\delta} = \tp{(0,0,6,6)}$, as shown in \fref{fig:proofn3:extangent}. The corresponding indices are labeled by $\xi, \delta_{12}, \delta_{13}, \delta_{23}$. In these cases, the hypergraphs are actually usual graphs.
\begin{figure}[!htbp]
	\centering
	\begin{subfigure}[b]{0.48\textwidth}
		\centering
		\includegraphics{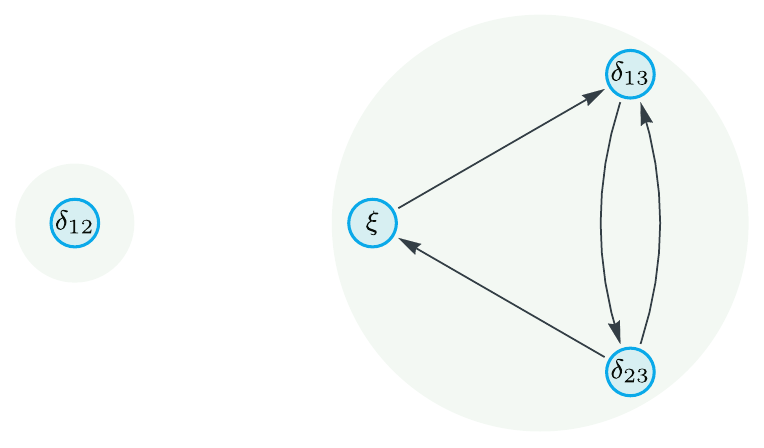}
	\end{subfigure}
	\hfill
	\begin{subfigure}[b]{0.48\textwidth}
		\centering
		\includegraphics{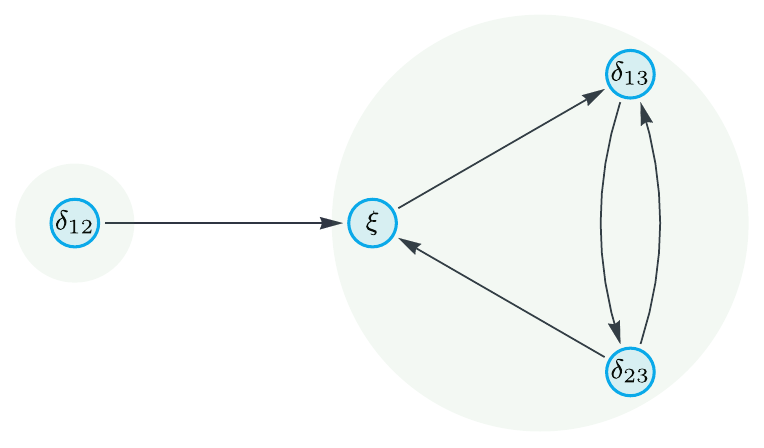}
	\end{subfigure}
	\caption{Examples of tangent directed hypergraph (left: $\tp{(0,2,6,6)}$, right: $\tp{(0,0,6,6)}$)}
	\label{fig:proofn3:extangent}
\end{figure}

\begin{definition}[Strongly connected components (SCC)]
	The strongly connected components $C$ of a directed hypergraph $\mathcal{H}$ are the equivalence classes of the relation $\equiv_{\mathcal{H}}$, defined by $u \equiv_{\mathcal{H}} v$ if $u \rightsquigarrow_{\mathcal{H}} v$ and $v \rightsquigarrow_{\mathcal{H}}  u$. SCCs are partially ordered by the relation $\preceq_{\mathcal{H}}$ induced by reachability: $C \preceq_{\mathcal{H}} C'$ if $C$ and $C'$ admit a representative $u$ and $u'$ such that $u\rightsquigarrow_{\mathcal{H}} u'$.
\end{definition}
\begin{theorem}[Allamigeon, Gaubert \& Goubault \cite{allamigeon2013computing}]\label{thm:all13}
	Let $\cone$ be a tropical polyhedral cone. A vector $\bm{v} \in \cone$ is tropically extreme if, and only if, the set of the strongly connected components of the tangent directed hypergraph at $\bm{v}$ in $\cone$, partially ordered by the reachability relation, admits a greatest element.
\end{theorem}

\noindent
In \fref{fig:proofn3:extangent}, SCCs are indicated by shading. The non-extreme one on the left has two disconnected SCCs, so it does not have a greatest one. One the other hand, the extreme one has a greatest SCC $\{ \xi, \delta_{13}, \delta_{23} \}$.
Now we are ready to prove Theorem \ref{maintheorem} for the case $n=3$.
\begin{proof}
Without loss of generality, assume $d_{12} \leq d_{13} \leq d_{23}$ (this can be done by relabeling the three items). The one $\ell^\infty$-nearest ultrametric given by Theorem \ref{thm:oneultra} is
\begin{equation}\label{proofn3:eq1}
\delta^* = \left(
\begin{array}{ccc}
0&d_{12} + q & d_{13} + q \\
d_{12} + q & 0 & d_{13} + q \\
d_{13} + q & d_{13} + q & 0 \\
\end{array}
\right),
\end{equation}
where $q = (d_{23}-d_{13})/2$.
We discuss by cases in the following.

\noindent
\textbf{Case 1:} First we consider $d_{12} < d_{13} < d_{23}$. Since $d_{13}+q - d_{13} = q$ and $d_{13}+q - d_{23} = -q$, we know that the root node is immobile. On the other hand, we can slide down the other node to $d_{12}-q$ while still keep it nearest. So we have the following two candidate rays (one in the case of $d_{13}=d_{23}$):
\begin{equation}\label{proofn3:eq2}
\tilde{\bm{\delta}} = \left(
\begin{array}{c}
0\\
d_{12}+q\\
d_{13}+q\\
d_{13}+q\\
\end{array}
\right) \sepand
\tilde{\bm{\delta}}=
\left(
\begin{array}{c}
0\\
d_{12}-q\\
d_{13}+q\\
d_{13}+q\\
\end{array}
\right),
\end{equation}
where the first one has one mobile node and the second one has none. Recall the exterior description \eref{extdesc:ex:eq3}:
\begin{equation}
\left(
\begin{array}{cccc}
&0&& \\
&&0& \\
&&&0 \\
d_{12}&&&\\
d_{13}&&&\\
d_{23}&&&\\
&-q&&\\
&&-q&\\
&&&-q\\
\end{array}
\right) \tilde{\bm{\delta}}
\leq
\left(
\begin{array}{cccc}
&&0&0 \\
&0&&0 \\
&0&0& \\
&q&&\\
&&q&\\
&&&q\\
d_{12}&&&\\
d_{13}&&&\\
d_{23}&&&\\
\end{array}
\right) \tilde{\bm{\delta}}.\tag{\ref{extdesc:ex:eq3}}
\end{equation}
Their corresponding tangent directed hypergraphs are shown in \fref{fig:proofn3:case1}, where indices are labeled by $\xi, \delta_{12}, \delta_{13}, \delta_{23}$ and SCCs are indicated by shading. They all have a greatest SCC: $\{\delta_{12}\}$ for the first one and $\{\xi, \delta_{13},\delta_{23}\}$ the second one, which means that both of them are extreme rays.
\begin{figure}[!htbp]
	\centering
	\begin{subfigure}[b]{0.48\textwidth}
		\centering
		\includegraphics{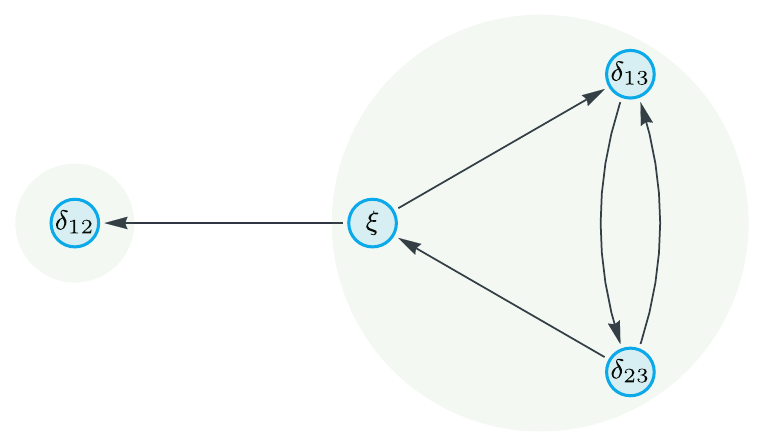}
	\end{subfigure}
	\hfill
	\begin{subfigure}[b]{0.48\textwidth}
		\centering
		\includegraphics{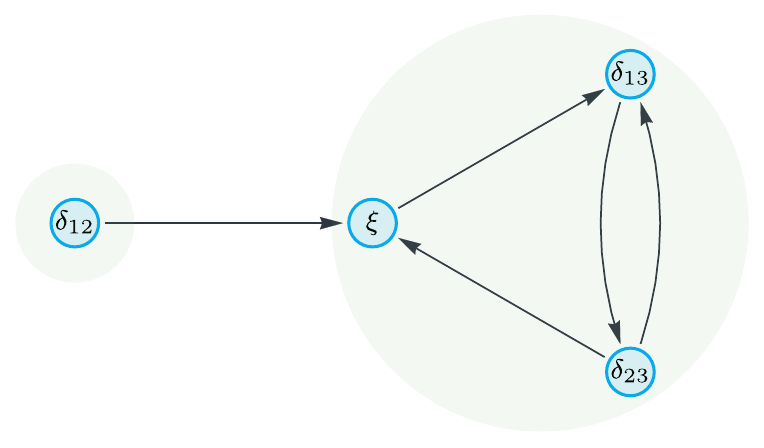}
	\end{subfigure}
	\caption{Tangent hypergraphs of case $d_{12}<d_{13}<d_{23}$}
	\label{fig:proofn3:case1}
\end{figure}

\noindent
\textbf{Case 2:} In the case of $d_{12}= d_{13} < d_{23}$, the tree is not binary. Each of the two resolutions corresponds to a candidate ray:
\begin{equation}\label{proofn3:eq3}
\tilde{\bm{\delta}} = \left(
\begin{array}{c}
0\\
d_{12}+q\\
d_{12}-q\\
d_{12}+q\\
\end{array}
\right) \sepand
\tilde{\bm{\delta}}=
\left(
\begin{array}{c}
0\\
d_{12}-q\\
d_{12}+q\\
d_{12}+q\\
\end{array}
\right).
\end{equation}
As shown in \fref{fig:proofn3:case2}, their corresponding tangent hypergraphs all have greatest SCC: $\{\xi,\delta_{12},\delta_{23}\}$ for the first one and $\{\xi,\delta_{13},\delta_{23}\}$ for the second one. So both of them are extreme rays.
\begin{figure}[!htbp]
	\centering
	\begin{subfigure}[b]{0.48\textwidth}
		\centering
		\includegraphics{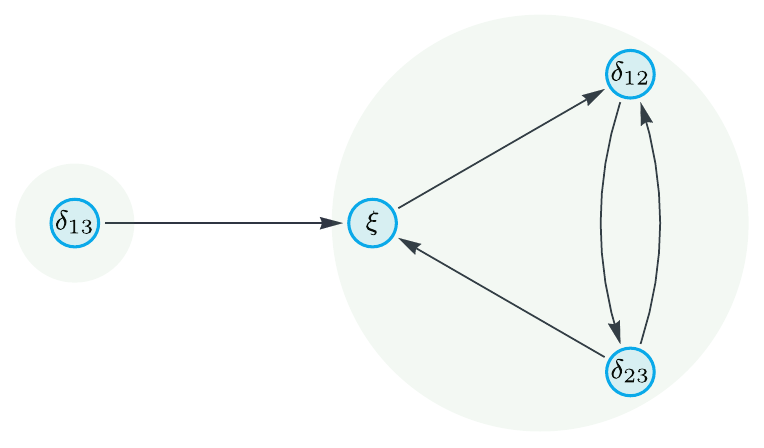}
	\end{subfigure}
	\hfill
	\begin{subfigure}[b]{0.48\textwidth}
		\centering
		\includegraphics{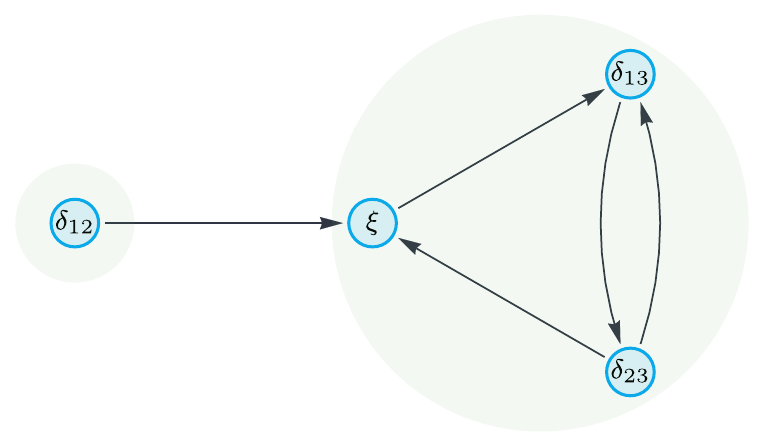}
	\end{subfigure}
	\caption{Tangent hypergraphs of case $d_{12}=d_{13}<d_{23}$}
	\label{fig:proofn3:case2}
\end{figure}

\noindent
\textbf{Case 3:} In the case of $d_{12}\leq d_{13} = d_{23}$, we have $q=0$. There is no freedom to slide the internal nodes. Thus, the only candidate ray is just the one given by Theorem \ref{thm:oneultra}:
\begin{equation}\label{proofn3:eq4}
\tilde{\bm{\delta}} = \left(
\begin{array}{c}
0\\
d_{12}\\
d_{13}\\
d_{13}\\
\end{array}
\right).
\end{equation}
Their corresponding tangent hypergraph is itself a SCC, thus it is greatest, indicating it extreme ray. The hypergraph is shown in \fref{fig:proofn3:case3}.
\begin{figure}[!htbp]
	\centering
	\includegraphics{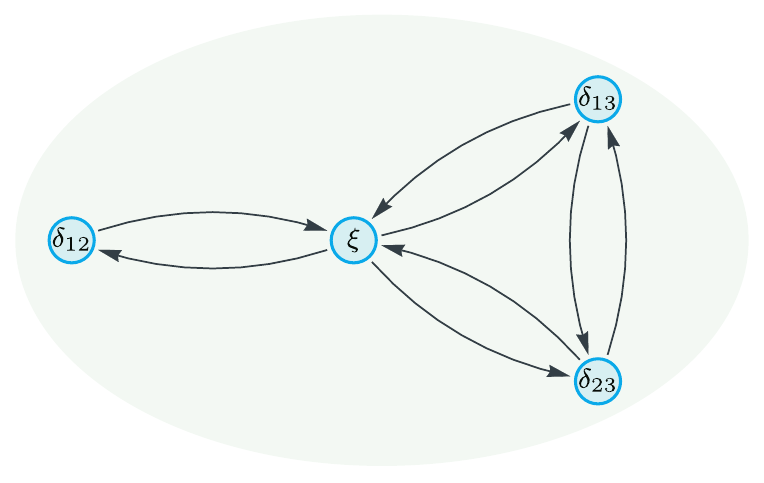}
	\caption{Tangent hypergraph of case $d_{12}\leq d_{13}=d_{23}$}
	\label{fig:proofn3:case3}
\end{figure}

With all possible cases discussed above, we can now conclude that Bernstein's characterization is sufficient when $n=3$. $\qed$
\end{proof}

\section{Proof of Theorem \ref{maintheorem} for the case $n \geq 4$}\label{sec:counter}
In this section, we show by counterexamples that the Bernstein's characterization \cite{bernstein2017} is not sufficient for $n\geq 4$. First we give a counterexample of $n=4$.
\subsection{An $n=4$ counterexample}
The given dissimilarity map and the one $\ell^\infty$-nearest ultrametric given by Theorem \ref{thm:oneultra} are
\begin{equation}\label{cntex2:eq1}
d = \left(
\begin{array}{cccc}
0 & 6 & 6 & 5 \\
6 & 0 & 14 & 12 \\
6 & 14 & 0 & 9 \\
5 & 12 & 9 & 0 \\
\end{array}
\right), \sep \delta^* = \left(
\begin{array}{cccc}
0 & 10 & 10 & 9 \\
10 & 0 & 10 & 10 \\
10 & 10 & 0 & 10 \\
9 & 10 & 10 & 0 \\
\end{array}
\right) \sepand q=4.
\end{equation}
By constructing the exterior description and using the extreme computing algorithm, we find 8 extreme rays, shown as columns in the following matrix.
\begin{equation}\label{cntex2:eq3}
\left(
\begin{array}{cccccccc}
10 & 10 & 8 & 2 & 2 & 10 & 10 & 9 \\
5 & 2 & 10 & 10 & 10 & 9 & 2 & 10 \\
1 & 5 & 1 & 8 & 9 & 9 & 9 & 9 \\
10 & 10 & 10 & 10 & 10 & 10 & 10 & 10 \\
10 & 10 & 8 & 8 & 9 & 10 & 10 & 8 \\
5 & 5 & 10 & 10 & 10 & 5 & 9 & 10 \\
\end{array}
\right).
\end{equation}
The corresponding trees are shown in \fref{fig:cntex2:extremes}.
\begin{figure}[!htbp]
	\centering
	\begin{subfigure}[b]{0.48\textwidth}
		\centering
		\includegraphics{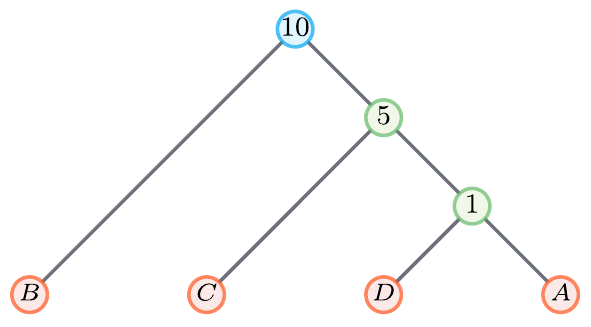}
	\end{subfigure}
	\hfill
	\begin{subfigure}[b]{0.48\textwidth}
		\centering
		\includegraphics{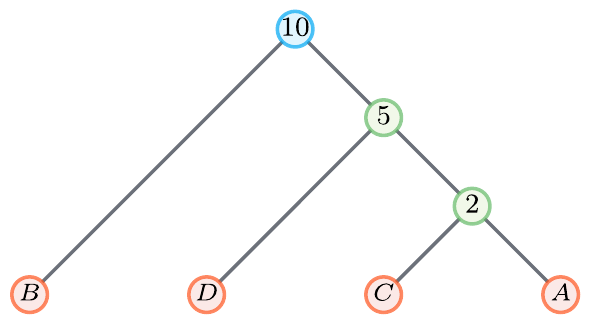}
	\end{subfigure}
	\\
	\begin{subfigure}[b]{0.48\textwidth}
		\centering
		\includegraphics{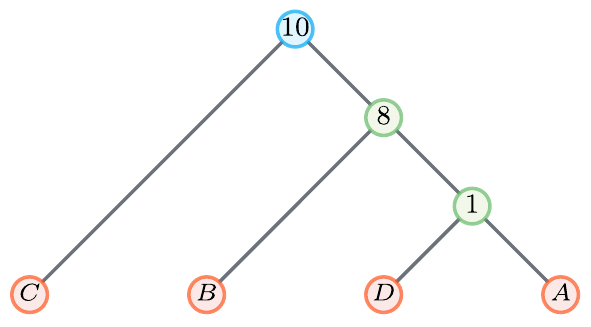}
	\end{subfigure}
	\hfill
	\begin{subfigure}[b]{0.48\textwidth}
		\centering
		\includegraphics{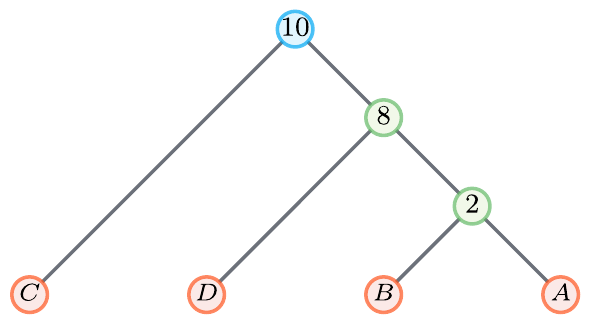}
	\end{subfigure}
	\\
	\begin{subfigure}[b]{0.48\textwidth}
		\centering
		\includegraphics{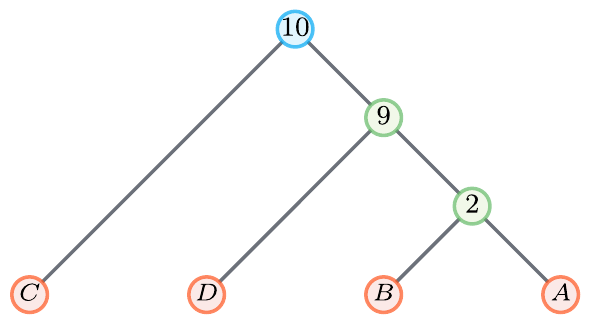}
	\end{subfigure}
	\hfill
	\begin{subfigure}[b]{0.48\textwidth}
		\centering
		\includegraphics{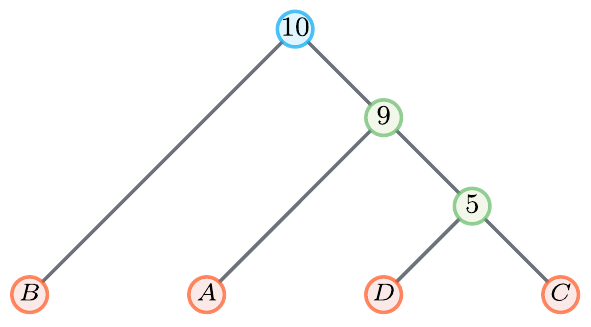}
	\end{subfigure}
	\\
	\begin{subfigure}[b]{0.48\textwidth}
		\centering
		\includegraphics{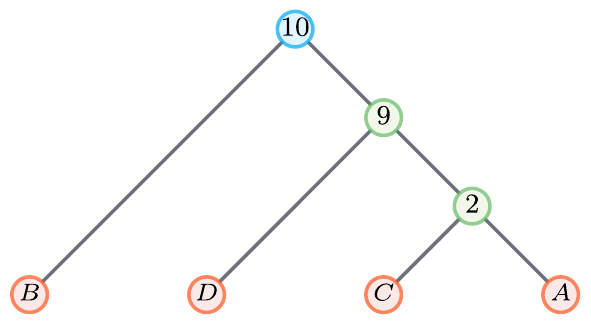}
	\end{subfigure}
	\hfill
	\begin{subfigure}[b]{0.48\textwidth}
		\centering
		\includegraphics{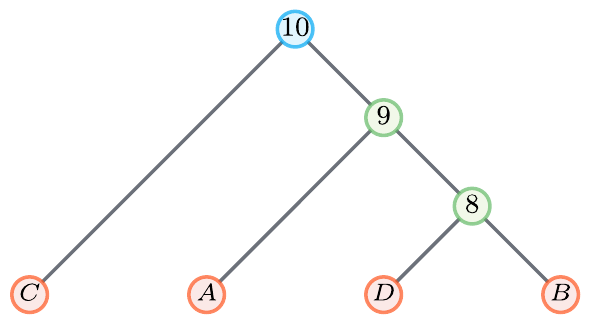}
	\end{subfigure}
	\caption{Extreme rays $\ell^\infty$-nearest to the dissimilarity map given by \eref{cntex2:eq1}}
	\label{fig:cntex2:extremes}
\end{figure}
However, the procedure in \cite{bernstein2017} will generate two more trees that also satisfy the characterization, as shown in \fref{fig:cntex2:more}.
\begin{figure}[!htbp]
	\centering
	\begin{subfigure}[b]{0.48\textwidth}
		\centering
		\includegraphics{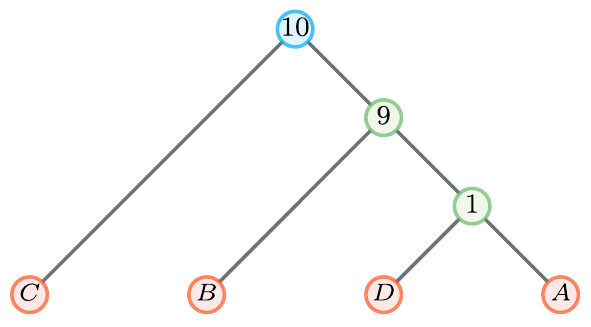}
	\end{subfigure}
	\hfill
	\begin{subfigure}[b]{0.48\textwidth}
		\centering
		\includegraphics{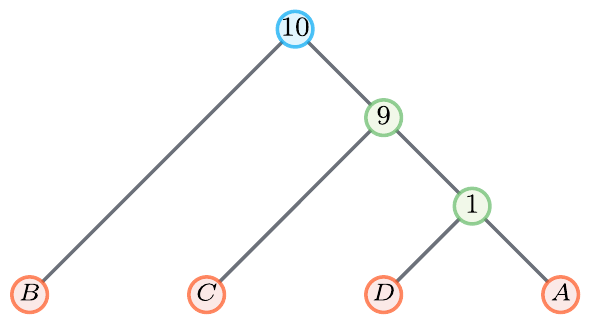}
	\end{subfigure}
	\caption{Non-extreme ultrametrics that satisfy the characterization}
	\label{fig:cntex2:more}
\end{figure}

\subsection{Construction of counterexamples for $n>4$}
We present a way to construct any $n>4$ ones inductively, i.e., size $n+1$ counterexample can be generated from size $n$ counterexample. First we prove the following lemma.
\begin{lemma}\label{lma:addnode}
	Suppose ${\delta} \in \mathbb{R}^{{n \choose 2 }}$ represents an $\ell^\infty$-nearest ultrametric to a given dissimilarity map $d$ on $n \geq 3$ elements. Define the extended dissimilarity map $d'$ on $n+1$ elements by $d'_{ij} = d_{ij}$ if $i, j \leq n$ and $d'_{i,n+1} = r +q+\epsilon$ if $i \leq n$ where $r$ is the root weight of ultrametric ${{\delta}}$ and $\epsilon >0$ is an arbitrary positive number. Define the extended ultrametric  ${{\delta}}' \in \mathbb{R}^{{n+1 \choose 2 }}$ by ${\delta}'_{ij}= {\delta}_{ij}$ if $i, j \leq n$ and ${\delta}'_{i,n+1} = r+\epsilon$. Then the extremality of ${{\delta}}'$ given $d'$ is the same as the extremality of ${{\delta}}$ given $d$, i.e., $\delta' \in \mathcal{E}_n(d') \Leftrightarrow \delta \in \mathcal{E}_n(d)$.
\end{lemma}
\begin{proof}
	First we perform the homogenization procedure on $\delta$ and $\delta'$ by letting $\xi=0$, getting $\tilde{\bm{\delta}}$ and $\tilde{\bm{\delta}}'$. Recall the exterior description of the $\ell^\infty$-nearest tropical cone \eref{extdesc:eq5}:
	\begin{equation}
	\begin{aligned}
	&\tilde{\delta}_{ij}\leq \max(\tilde{\delta}_{jk},\tilde{\delta}_{ik}),\\
	&\tilde{\delta}_{ij}-q \leq  \xi + d_{ij} \sepand \xi + d_{ij}  \leq \tilde{\delta}_{ij} + q.
	\end{aligned}\tag{\ref{extdesc:eq5}}
	\end{equation}
	Consider the difference between tangent directed hypergraphs built from $\tilde{\bm{\delta}}$ on $d$ and $\tilde{\bm{\delta}}'$ on $d'$. The hypergraph of $\tilde{\bm{\delta}}$ will have ${n \choose 2} +1$ nodes and that of $\tilde{\bm{\delta}}'$ will have ${n+1 \choose 2} +1$ nodes, where the $n$ extra nodes correspond to the $n$ variables $\tilde{\delta}_{i, n+1}$. An schematic diagram is shown in \fref{fig:proofn4:scheme}.
\begin{figure}[!htbp]
	\centering
	\includegraphics{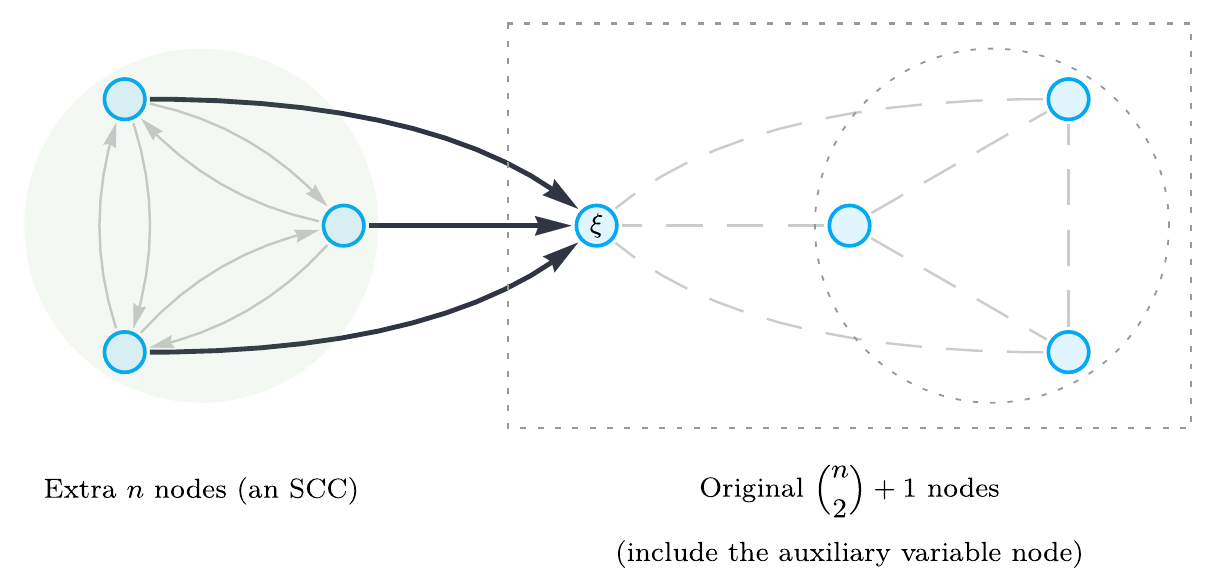}
	\caption{Schematic tangent hypergraph of size $n+1$ ultrametric $\tilde{\bm{\delta}}$}
	\label{fig:proofn4:scheme}
\end{figure}

 We focus on the extra hyperarcs between these $n$ extra nodes and the original ${n \choose 2} +1$ nodes. The first set of inequalities in \eref{extdesc:eq5} will not make any difference since any combination of $(i,j,n+1)$ with $i,j\leq n$ will only generate hyperarcs among the $n$ extra nodes themselves but not between the two groups, due to the fact that $\tilde{\delta}'_{i,n+1} = r+\epsilon > \tilde{\delta}'_{i,j}$ with $\forall i,j\leq n$. On the other hand, the second set of inequalities will contribute extra hyperarcs. Consider the inequalities $\xi + d_{ij} \leq \tilde{\delta}_{ij} + q$. Since ${\delta}'_{i,n+1} + q= r+\epsilon +q =d'_{i,n+1}$, extra hyperarcs from each of the $n$ extra nodes to the first node representing the homogenization variable $\xi$ will be generated. So the overall difference between these two hypergraphs is that an extra SCC consisting of $n$ extra nodes will be pointing towards the first node. Now it is obvious to see that this will not change the fact whether or not the hypergraph has a greatest SCC since this newly added SCC cannot be the greatest one. Therefore by Theorem \ref{thm:all13}, the extremality of ${{\delta}}'$ given $d'$ is the same as the extremality of ${{\delta}}$ given $d$.  $\qed$
\end{proof}
Lemma \ref{lma:addnode} tells us that once the ultrametric is not extreme, there will always be a way to construct another larger non-extreme ultrametric upon it.  So we have the following theorem:
\begin{theorem}\label{thm:induc}
	Let $d^{(n)}$ be a dissimilarity map on $n \geq 4$ elements. If $d^{(n)}$ is a counterexample to Bernstein's characterization, then there exists $d^{(n+1)}$ on $n+1$ elements that is also a counterexample.
\end{theorem}
\begin{proof}
	Suppose $\delta^{(n)} \in \tmr^{{n\choose 2}}$ is an ultrametric produced by Bernstein's Theorem \ref{thm:charac} but not extreme. We construct $d^{(n+1)}$ and $\delta^{(n+1)}$ in the same way as in Lemma \ref{lma:addnode}, based on $d^{(n)}$ and $\delta^{(n)}$. Since $\delta^{(n)}$ is not extreme, according to Lemma \ref{lma:addnode}, $\delta^{(n+1)}$ is also not an extreme ray. Now we only have to show $\delta^{(n+1)}$ satisfies Bernstein's characterization, i.e., can be generated from the sliding procedure and has at most one mobile node.
	
	First, for $d^{(n+1)}$, the one $\ell^\infty$-nearest ultrametric given by Theorem \ref{thm:oneultra} will only differ from that of $d^{(n)}$ when involving the $(n+1)$-th element. This is because by construction, the distance from any element to $(n+1)$-th element is the largest. We have $\delta^{*(n+1)}_{ij} = \delta^{*(n)}_{ij}$ if $i,j\leq n$ and $\delta^{*(n+1)}_{i,n+1} = r+ 2q+\epsilon$. The tree constructed from $\delta^{*(n+1)}$ will have a root node with weight $r+2q+\epsilon$, whose two children are the $(n+1)$-th element and the entire tree constructed from $\delta^{*(n)}$. Note that the root node is mobile and we can slide the root node all the way done till $r+\epsilon$. Now the root node is immobile and has a greater weight than the root of $\delta^{*(n)}$. So further sliding procedures will only take place on the $\delta^{*(n)}$ sub-tree. This means that all possible trees generated by sliding $\delta^{*(n)}$ will also appear, as a child of the root node whose weight is $r+\epsilon$. Moreover, they should all have at most one mobile node since the root itself is immobile. Because the above $\delta^{(n)}$ is generated from $\delta^{*(n)}$, $\delta^{(n+1)}$ defined as in Lemma \ref{lma:addnode} should appear in the sliding procedure and has at most one mobile node, hence satisfying Bernstein's characterization. $\qed$
\end{proof}
Now we can complete the proof of Theorem \ref{maintheorem}.
\begin{proof}
	Given the $n=4$ counterexample \eref{cntex2:eq1}, we can inductively construct any size $n(\geq 4)$ counterexample according to Theorem \ref{thm:induc}. $\qed$
\end{proof}

As an example, here we show how to construct an $n=5$ counterexample based on the $n=4$ counterexample in \eref{cntex2:eq1}. We will focus on the left non-extreme ultrametric in \fref{fig:cntex2:more}. Since the root weight is $r=10$ and $q=4$. We choose $\epsilon=1$ then we can write the following $n=5$ dissimilarity map
\begin{equation}\label{constr:eq1}
d = \left(
\begin{array}{ccccc}
0 & 6 & 6 & 5 & 15\\
6 & 0 & 14 & 12 & 15 \\
6 & 14 & 0 & 9 & 15 \\
5 & 12 & 9 & 0 & 15 \\
15 & 15 & 15 & 15 & 0\\
\end{array}
\right)
\end{equation}
and the one $\ell^\infty$-nearest ultrametric given by Theorem \ref{thm:oneultra} is
\begin{equation}\label{constr:eq2}
\delta^* = \left(
\begin{array}{ccccc}
0 & 10 & 10 & 9 & 19\\
10 & 0 & 10 & 10 & 19 \\
10 & 10 & 0 & 10 & 19 \\
9 & 10 & 10 & 0 & 19\\
19 & 19 & 19 & 19 & 0\\
\end{array}
\right).
\end{equation}
It is straightforward to check that the following ultrametric shown in \fref{fig:constr} with root weight $q+\epsilon=11$ satisfies Bernstein's characterization but is not an extreme ray.
\begin{figure}[!htbp]
	\centering
	\includegraphics{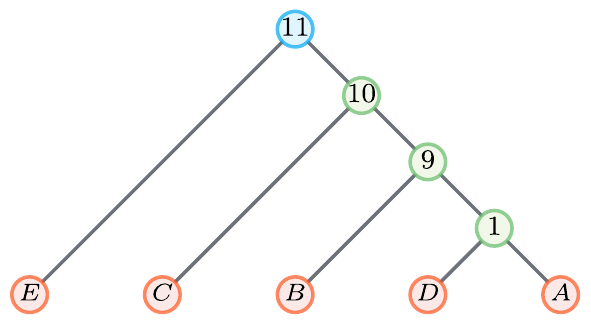}
	\caption{Non-extreme ultrametrics constructed from the left in \fref{fig:cntex2:more} satisfies the characterization}
	\label{fig:constr}
\end{figure}

\subsection{A counterexample based on the biological data in \cite{bernstein2017}}
The counterexample given above is far from unique. To show this, we examine the $n=8$ example given in \cite{bernstein2017}.  The given dissimilarity map is
\begin{equation}\label{cntex:eq1}
d = \left(
\begin{array}{cccccccc}
 0 & 32 & 48 & 51 & 50 & 48 & 98 & 148 \\
 32 & 0 & 26 & 34 & 29 & 33 & 84 & 136 \\
 48 & 26 & 0 & 42 & 44 & 44 & 92 & 152 \\
 51 & 34 & 42 & 0 & 44 & 38 & 86 & 142 \\
 50 & 29 & 44 & 44 & 0 & 24 & 89 & 142 \\
 48 & 33 & 44 & 38 & 24 & 0 & 90 & 142 \\
 98 & 84 & 92 & 86 & 89 & 90 & 0 & 148 \\
 148 & 136 & 152 & 142 & 142 & 142 & 148 & 0 \\
\end{array}
\right).
\end{equation}
This dataset represents immunological distances between dot(D), bear(B), raccoon(R), weasel(W), seal(S), sea lion(SL), cat(C), monkey(M). The one $\ell^\infty$-nearest ultrametric given by Theorem \ref{thm:oneultra} is
\begin{equation}\label{cntex:eq2}
\delta^* = \left(
\begin{array}{cccccccc}
 0 & 41 & 41 & 43 & 41 & 41 & 93 & 145 \\
 41 & 0 & 35 & 43 & 38 & 38 & 93 & 145 \\
 41 & 35 & 0 & 43 & 38 & 38 & 93 & 145 \\
 43 & 43 & 43 & 0 & 43 & 43 & 93 & 145 \\
 41 & 38 & 38 & 43 & 0 & 33 & 93 & 145 \\
 41 & 38 & 38 & 43 & 33 & 0 & 93 & 145 \\
 93 & 93 & 93 & 93 & 93 & 93 & 0 & 145 \\
 145 & 145 & 145 & 145 & 145 & 145 & 145 & 0 \\
\end{array}
\right).
\end{equation}
This ultrametric corresponds to the tree in \fref{fig:cntex:oneultra}. The $\ell^\infty$-nearest distance $q = 9$.
\begin{figure}[!htbp]
	\centering
	\includegraphics{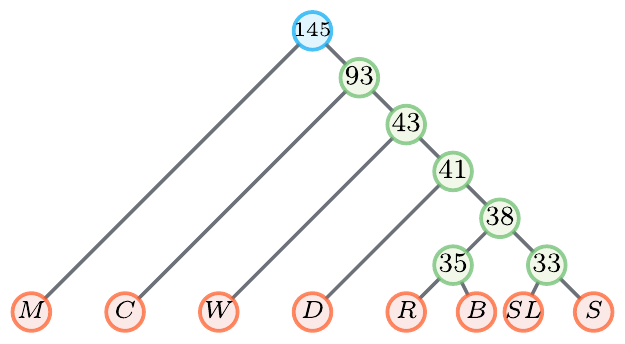}
	\caption{One ultrametric returned by Theorem \ref{thm:oneultra}}
	\label{fig:cntex:oneultra}
\end{figure}
We can then construct the exterior description, which will be matrices of size $224\times 29$. By the extreme computing algorithm, 16 extreme rays are found as following, where each column represents an extreme ray.
\begin{equation}\label{cntex:eq3}
\left(
\begin{array}{cccccccccccccccc}
 41 & 41 & 41 & 41 & 41 & 41 & 41 & 41 & 41 & 41 & 41 & 41 & 41 & 41 & 41 & 41 \\
 41 & 41 & 41 & 41 & 41 & 41 & 41 & 41 & 41 & 41 & 41 & 41 & 41 & 41 & 41 & 41 \\
 42 & 42 & 42 & 43 & 42 & 42 & 42 & 43 & 42 & 42 & 42 & 43 & 42 & 42 & 42 & 42 \\
 41 & 41 & 41 & 41 & 41 & 41 & 41 & 41 & 41 & 41 & 41 & 41 & 41 & 41 & 41 & 41 \\
 41 & 41 & 41 & 41 & 41 & 41 & 41 & 41 & 41 & 41 & 41 & 41 & 41 & 41 & 41 & 41 \\
 89 & 89 & 89 & 89 & 93 & 89 & 89 & 89 & 93 & 89 & 89 & 89 & 93 & 89 & 89 & 89 \\
 143 & 143 & 143 & 143 & 143 & 143 & 145 & 143 & 143 & 143 & 145 & 143 & 143 & 143 & 143 & 145 \\
 35 & 35 & 17 & 35 & 35 & 35 & 35 & 35 & 35 & 35 & 35 & 17 & 17 & 17 & 17 & 17 \\
 42 & 42 & 42 & 43 & 42 & 42 & 42 & 43 & 42 & 42 & 42 & 43 & 42 & 42 & 42 & 42 \\
 24 & 20 & 35 & 24 & 24 & 33 & 24 & 20 & 20 & 20 & 20 & 35 & 35 & 38 & 35 & 35 \\
 24 & 24 & 35 & 24 & 24 & 24 & 24 & 24 & 24 & 33 & 24 & 35 & 35 & 38 & 35 & 35 \\
 89 & 89 & 89 & 89 & 93 & 89 & 89 & 89 & 93 & 89 & 89 & 89 & 93 & 89 & 89 & 89 \\
 143 & 143 & 143 & 143 & 143 & 143 & 145 & 143 & 143 & 143 & 145 & 143 & 143 & 143 & 143 & 145 \\
 42 & 42 & 42 & 43 & 42 & 42 & 42 & 43 & 42 & 42 & 42 & 43 & 42 & 42 & 42 & 42 \\
 35 & 35 & 35 & 35 & 35 & 35 & 35 & 35 & 35 & 35 & 35 & 35 & 35 & 38 & 35 & 35 \\
 35 & 35 & 35 & 35 & 35 & 35 & 35 & 35 & 35 & 35 & 35 & 35 & 35 & 38 & 35 & 35 \\
 89 & 89 & 89 & 89 & 93 & 89 & 89 & 89 & 93 & 89 & 89 & 89 & 93 & 89 & 89 & 89 \\
 143 & 143 & 143 & 143 & 143 & 143 & 145 & 143 & 143 & 143 & 145 & 143 & 143 & 143 & 143 & 145 \\
 42 & 42 & 42 & 43 & 42 & 42 & 42 & 43 & 42 & 42 & 42 & 43 & 42 & 42 & 42 & 42 \\
 42 & 42 & 42 & 43 & 42 & 42 & 42 & 43 & 42 & 42 & 42 & 43 & 42 & 42 & 42 & 42 \\
 89 & 89 & 89 & 89 & 93 & 89 & 89 & 89 & 93 & 89 & 89 & 89 & 93 & 89 & 89 & 89 \\
 143 & 143 & 143 & 143 & 143 & 143 & 145 & 143 & 143 & 143 & 145 & 143 & 143 & 143 & 143 & 145 \\
 15 & 24 & 15 & 15 & 15 & 33 & 15 & 24 & 24 & 33 & 24 & 15 & 15 & 15 & 33 & 15 \\
 89 & 89 & 89 & 89 & 93 & 89 & 89 & 89 & 93 & 89 & 89 & 89 & 93 & 89 & 89 & 89 \\
 143 & 143 & 143 & 143 & 143 & 143 & 145 & 143 & 143 & 143 & 145 & 143 & 143 & 143 & 143 & 145 \\
 89 & 89 & 89 & 89 & 93 & 89 & 89 & 89 & 93 & 89 & 89 & 89 & 93 & 89 & 89 & 89 \\
 143 & 143 & 143 & 143 & 143 & 143 & 145 & 143 & 143 & 143 & 145 & 143 & 143 & 143 & 143 & 145 \\
 143 & 143 & 143 & 143 & 143 & 143 & 145 & 143 & 143 & 143 & 145 & 143 & 143 & 143 & 143 & 145 \\
\end{array}
\right)
\end{equation}
And the corresponding trees are shown in \fref{fig:cntex:extremes}. The indices under the trees indicate the order that they are arranged in \cite{bernstein2017}. Note that in (f), the node with weight 89 has weight 93 in \cite{bernstein2017}, which might be a typo in the original paper. There are 4 more non-extreme ultrametrics that fail the characterization, shown in \fref{fig:cntex:more}.
\begin{figure}[!htbp]
	\centering
	\begin{subfigure}[b]{0.48\textwidth}
        \centering
		\includegraphics{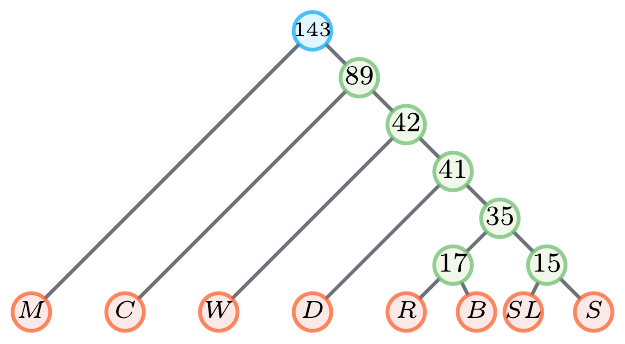}
		\caption{No. 1}
	\end{subfigure}
    \hfill
    \begin{subfigure}[b]{0.48\textwidth}
        \centering
		\includegraphics{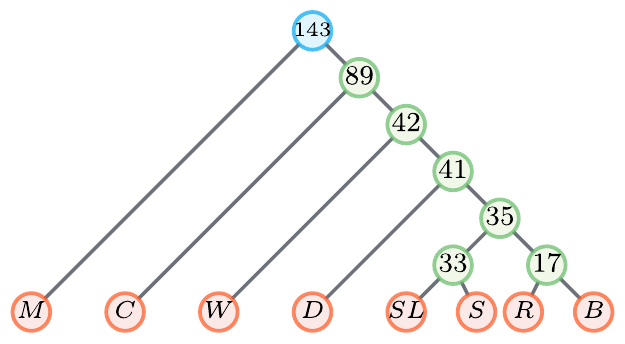}
		\caption{No. 3}
	\end{subfigure}
\\
    \begin{subfigure}[b]{0.48\textwidth}
        \centering
		\includegraphics{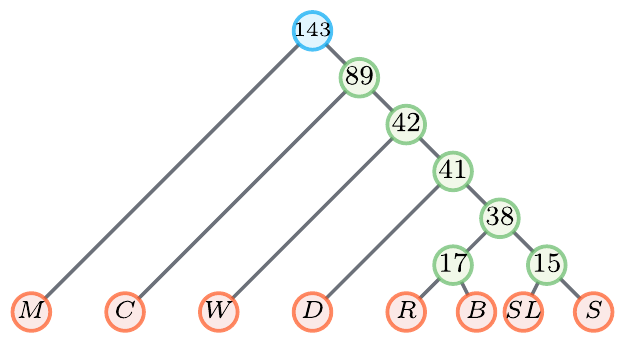}
		\caption{No. 4}
	\end{subfigure}
    \hfill
	\begin{subfigure}[b]{0.48\textwidth}
        \centering
		\includegraphics{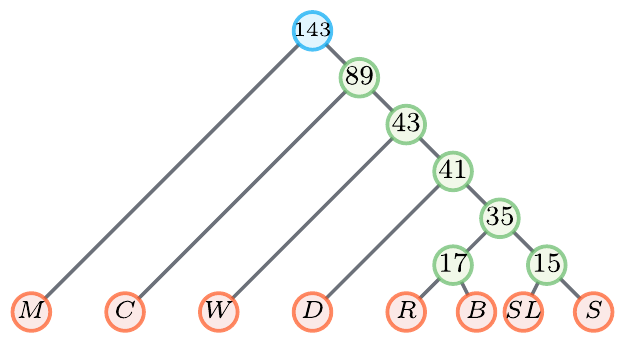}
		\caption{No. 5}
	\end{subfigure}
	\\
    \begin{subfigure}[b]{0.48\textwidth}
        \centering
		\includegraphics{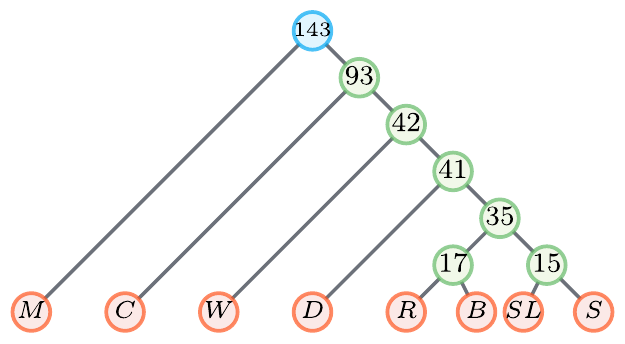}
		\caption{No. 6}
	\end{subfigure}
	\hfill
    \begin{subfigure}[b]{0.48\textwidth}
        \centering
		\includegraphics{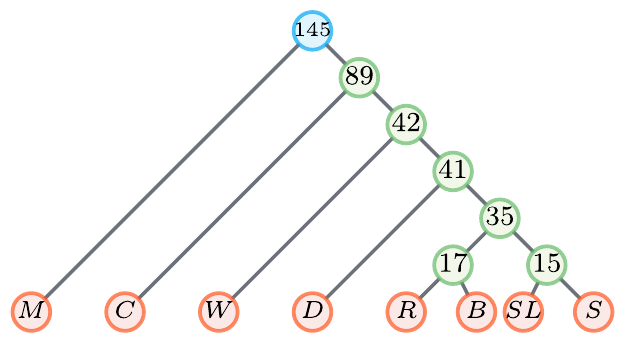}
		\caption{No. 7 (*)}
	\end{subfigure}
    \\
	\begin{subfigure}[b]{0.48\textwidth}
        \centering
		\includegraphics{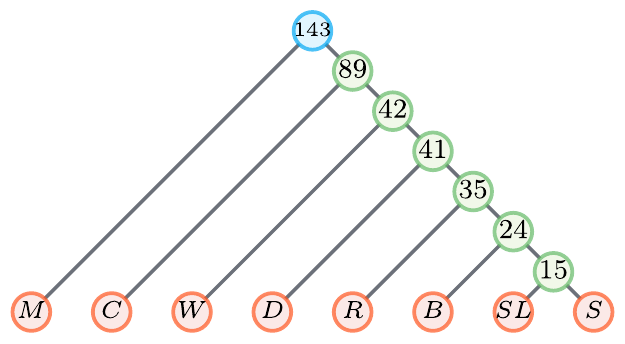}
		\caption{No. 8}
	\end{subfigure}
	\hfill
    \begin{subfigure}[b]{0.48\textwidth}
        \centering
		\includegraphics{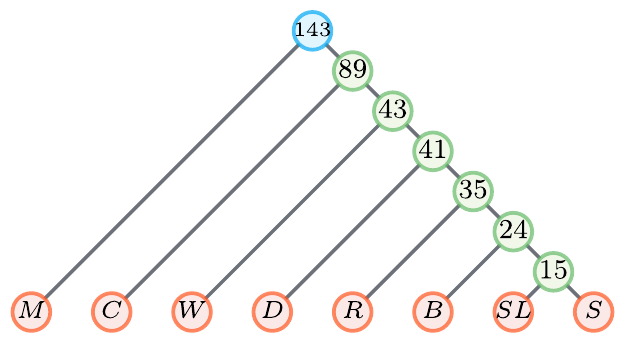}
		\caption{No. 11}
	\end{subfigure}
    \\
	\begin{subfigure}[b]{0.48\textwidth}
        \centering
		\includegraphics{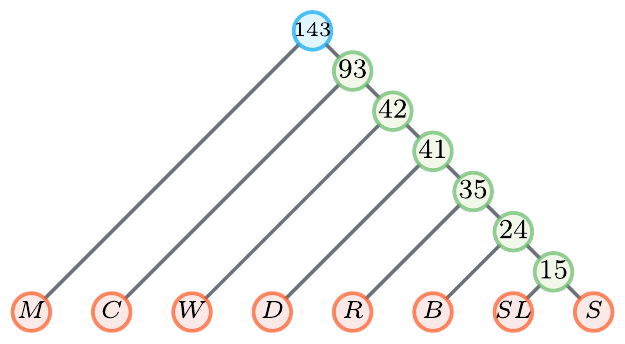}
		\caption{No. 12}
	\end{subfigure}
	\hfill
	\begin{subfigure}[b]{0.48\textwidth}
        \centering
		\includegraphics{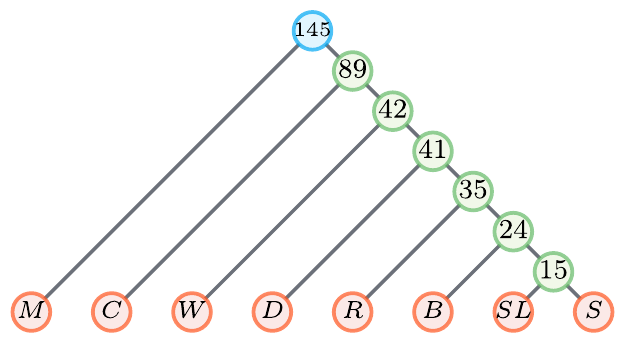}
		\caption{No. 13}
	\end{subfigure}
        \caption{Extreme rays $\ell^\infty$-nearest to the dissimilarity map given by \eref{cntex:eq1} (continue next page)}
	\label{fig:cntex:extremes}
\end{figure}
\begin{figure}[!htbp]\ContinuedFloat
	\centering
	\begin{subfigure}[b]{0.48\textwidth}
        \centering
		\includegraphics{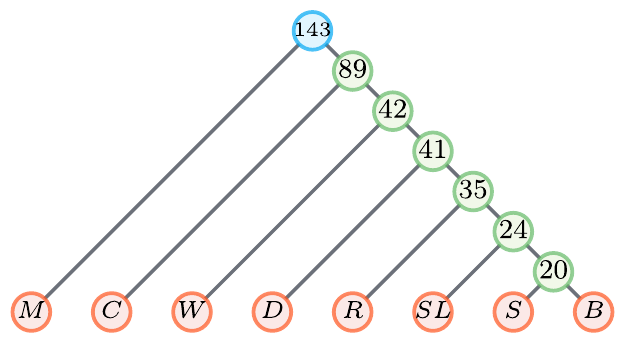}
		\caption{No. 14}
	\end{subfigure}
    \hfill
	\begin{subfigure}[b]{0.48\textwidth}
        \centering
		\includegraphics{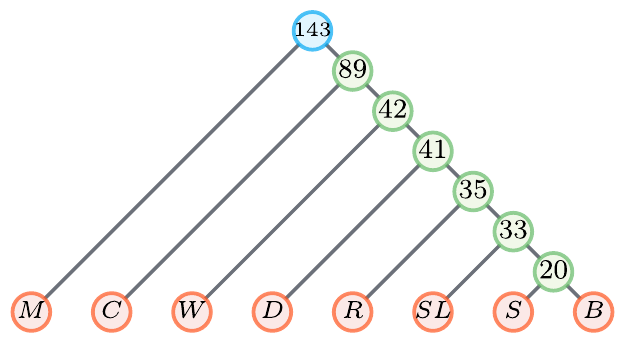}
		\caption{No. 15}
	\end{subfigure}
    \\
	\begin{subfigure}[b]{0.48\textwidth}
        \centering
		\includegraphics{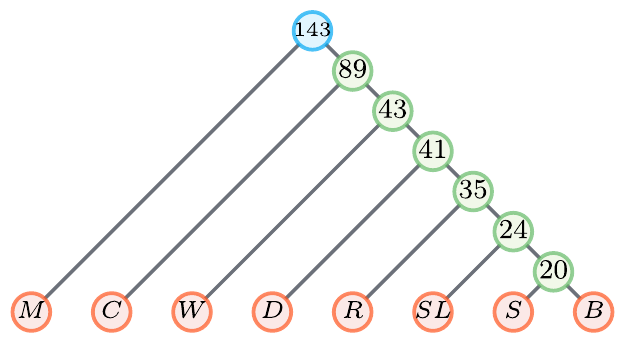}
		\caption{No. 17}
	\end{subfigure}
	\hfill
	\begin{subfigure}[b]{0.48\textwidth}
        \centering
		\includegraphics{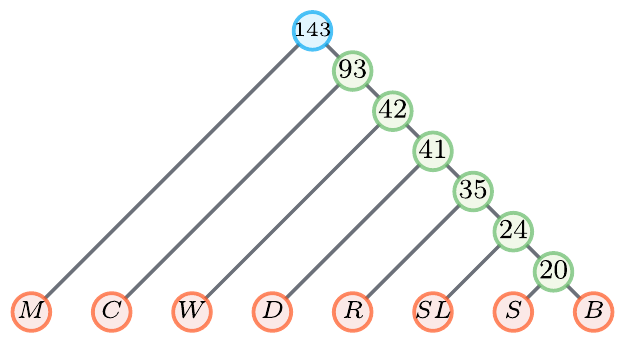}
		\caption{No. 18}
	\end{subfigure}
	\\
	\begin{subfigure}[b]{0.48\textwidth}
        \centering
		\includegraphics{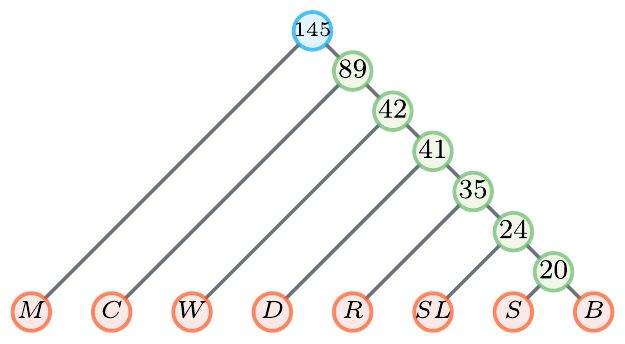}
		\caption{No. 19}
	\end{subfigure}
	\hfill
    \begin{subfigure}[b]{0.48\textwidth}
        \centering
		\includegraphics{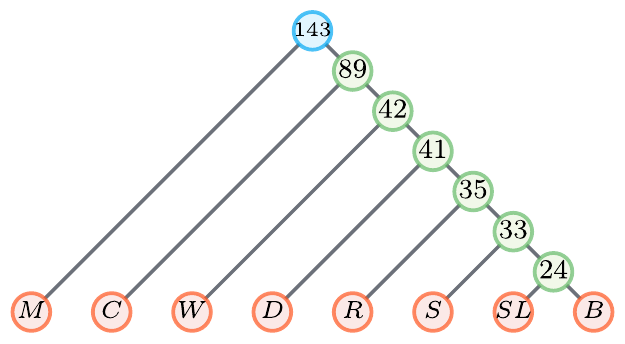}
		\caption{No. 20}
	\end{subfigure}
	\caption{Extreme rays $\ell^\infty$-nearest to the dissimilarity map given by \eref{cntex:eq1}}
\end{figure}

\begin{figure}[!htbp]
	\centering
	\begin{subfigure}[b]{0.48\textwidth}
        \centering
		\includegraphics{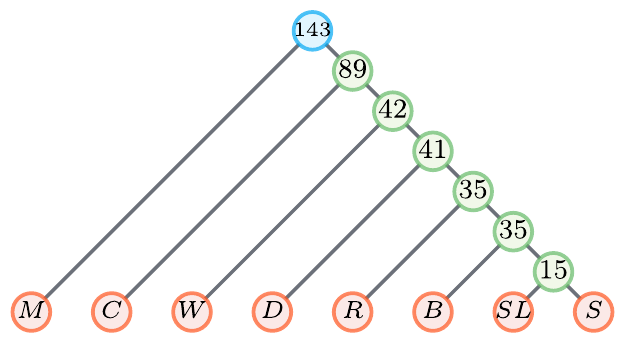}
		\caption{No. 2}
	\end{subfigure}
	\hfill
    \begin{subfigure}[b]{0.48\textwidth}
        \centering
		\includegraphics{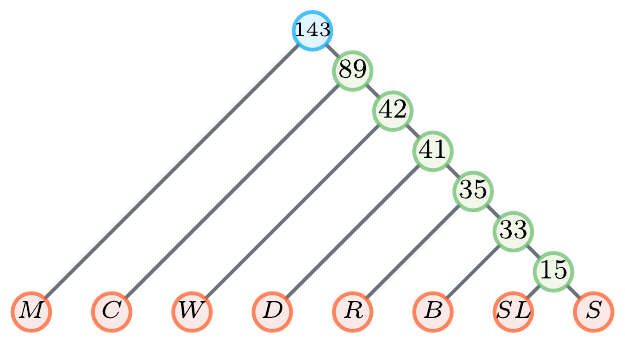}
		\caption{No. 9}
	\end{subfigure}
    \\
    \begin{subfigure}[b]{0.48\textwidth}
        \centering
		\includegraphics{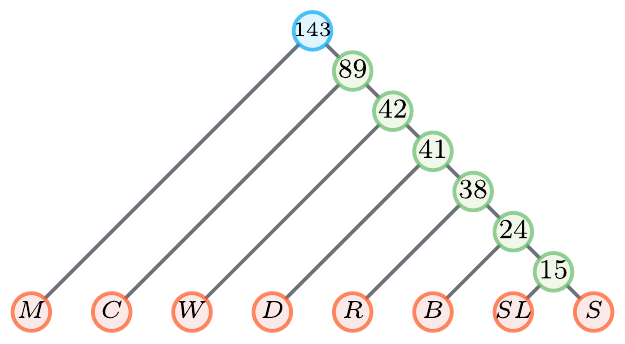}
		\caption{No. 10}
	\end{subfigure}
	\hfill
    \begin{subfigure}[b]{0.48\textwidth}
        \centering
		\includegraphics{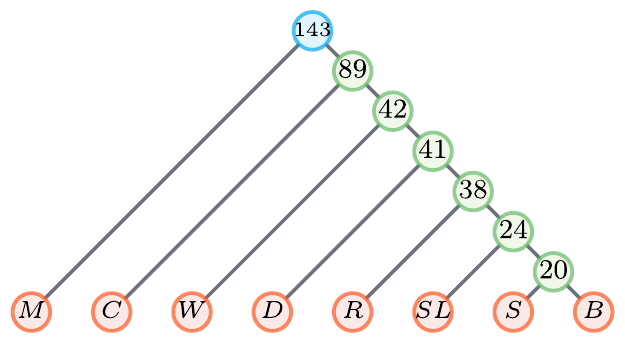}
		\caption{No. 16}
	\end{subfigure}
	\caption{Non-extreme ultrametrics that satisfy the characterization}
	\label{fig:cntex:more}
\end{figure}

\section{Summary}\label{sec:disc}
In this paper, we prove that Bernstein's characterization \cite{bernstein2017} for extreme rays of the tropical polytope formed by ultrametrics $\ell^\infty$-nearest to a given dissimilarity map is sufficient if and only if $n=3$. We first give the exterior description of this tropical polytope, which can be used not only for the purpose of this proof, but also to compute all extreme rays of the polytope by algorithm proposed in \cite{allamigeon2013computing}. The proof utilizes the tangent directed hypergraph technique\cite{allamigeon2014complexity} which determines the extremality of a vector by constructing the tangent directed hypergraph based on the exterior description of the tropical polytope. The $n=3$ sufficiency is proved directly by this technique and the $n\geq 4$ insufficiency is proved inductively.

There are two important open questions. First, is there a direct characterization on phylogenetic trees that distinguishes extreme and non-extreme ones? One possibility, for example, might be adding some additional constraints to Bernstein's characterization. Second, is there an efficient algorithm to generate all extreme rays of out particular polytope? The existing algorithm by Allamigeon et al. for general cases grows exponentially with $n$ since the rows and columns of the matrix are both increasing.

\section*{Acknowledgement}\label{sec:ack}
The author was supported by the Provost’s Graduate Excellence Fellowships at University of Texas at Austin. The author would like to thank Ngoc Mai Tran for the initiation of this project and her continuous guidance and insightful inputs. Also, the author would like to thank Xavier Allamigeon for his advice on the manuscript and constructive discussions, and Daniel Irving Bernstein for his detailed explanations and clarifications about his theorems.

\bibliographystyle{unsrt}
\bibliography{TM-Paper}

\end{document}